\documentclass{article}

\usepackage[nonatbib,preprint]{nips_2018}
\usepackage{graphicx}
\usepackage{amsmath,amssymb,amsfonts,amstext,amsthm,mathrsfs}
\usepackage[utf8]{inputenc} % allow utf-8 input
\usepackage[T1]{fontenc}    % use 8-bit T1 fonts
\usepackage{hyperref}       % hyperlinks
\usepackage{url}            % simple URL typesetting
\usepackage{booktabs}       % professional-quality tables
\usepackage{nicefrac}       % compact symbols for 1/2, etc.
\usepackage{microtype}      % microtypography
\usepackage{cleveref}
\usepackage{dsfont}
\usepackage{enumitem}
\usepackage{thm-restate}
\usepackage{color}

\newtheorem{definition}{Definition}

\newtheorem{thm}{Theorem}
\newtheorem{coro}{Corollary}
\newtheorem{assum}{Assumption}

\makeatletter
\newcommand*\bigcdot{\mathpalette\bigcdot@{.5}}
\newcommand*\bigcdot@[2]{\mathbin{\vcenter{\hbox{\scalebox{#2}{$\m@th#1\bullet$}}}}}
\makeatother

\newcommand{\xb}{\mathbf{x}}
\newcommand{\yb}{\mathbf{y}}
\newcommand{\zb}{\mathbf{z}}
\newcommand{\wb}{\mathbf{w}}
\newcommand{\ub}{\mathbf{u}}

\newcommand{\RR}{\mathds{R}}
\newcommand{\dist}{\mathrm{dist}}

\newcommand{\zero}{\mathbf{0}}

\newcommand{\argmin}{\mathop{\mathrm{argmin}}}
\newcommand{\KL}{K{\L}~}
\newcommand{\ttop}{{\!\top}}
\newcommand{\inner}[2]{\langle #1, #2 \rangle}

\title{Convergence of Cubic Regularization for Nonconvex Optimization under \KL Property}

\author{
  Yi Zhou \\
  Department of ECE\\
  The Ohio State University\\
  \texttt{zhou.1172@osu.edu} \\
   \And
   Zhe Wang \\
   Department of ECE\\
   The Ohio State University\\
   \texttt{wang.10982@osu.edu} \\
   \AND
   Yingbin Liang \\
   Department of ECE\\
   The Ohio State University\\
   \texttt{liang.889@osu.edu} \\
}

\begin{document}

\maketitle

\begin{abstract}
  Cubic-regularized Newton's method (CR) is a popular algorithm that guarantees to produce a second-order stationary solution  for solving nonconvex optimization problems. However, existing understandings of the convergence rate of CR are conditioned on special types of geometrical properties of the objective function. In this paper, we explore the asymptotic convergence rate of CR by exploiting the ubiquitous Kurdyka-{\L}ojasiewicz (\KL\!\!) property of nonconvex objective functions. In specific, we characterize the asymptotic convergence rate of various types of optimality measures for CR including function value gap, variable distance gap, gradient norm and least eigenvalue of the Hessian matrix. Our results fully characterize the diverse convergence behaviors of these optimality measures in the full parameter regime of the \KL property. Moreover, we show that the obtained asymptotic convergence rates of CR are order-wise faster than those of first-order gradient descent algorithms under the \KL property.
\end{abstract}

\section{Introduction}
A majority of machine learning applications are naturally formulated as nonconvex optimization due to the complex mechanism of the underlying model. Typical examples include training neural networks in deep learning \cite{Goodfellow2016}, low-rank matrix factorization \cite{RongGe2016,Bhojanapalli2016}, phase retrieval \cite{Candes2015,Zhang2017}, etc. In particular, these machine learning problems take the generic form
\begin{align*}
\min_{\xb\in \RR^d}~~f(\xb), \tag{P}
\end{align*}
where the objective function $f: \RR^d \to \RR$ is a differentiable and nonconvex function, and usually corresponds to the total loss over a set of data samples. 
Traditional first-order algorithms such as gradient descent can produce a solution $\bar{\xb}$ that satisfies the first-order stationary condition, i.e., $\nabla f (\bar{\xb}) = \zero$. However, such first-order stationary condition does not exclude the possibility of approaching a saddle point, which can be a highly suboptimal solution and therefore deteriorate the performance. 

Recently in the machine learning community, there has been an emerging interest in designing algorithms to scape saddle points in nonconvex optimization, and one popular algorithm is the cubic-regularized Newton's algorithm \cite{Nesterov2006,Agarwal2017,Yue2018}, which is also referred to as cubic regularization (CR) for simplicity. Such a second-order method exploits the Hessian information and produces a solution $\bar{\xb}$ for problem (P) that satisfies the second-order stationary condition, i.e., 
$$(\text{second-order stationary}):\quad \nabla f (\bar{\xb}) = \zero, ~\nabla^2 f(\bar{\xb}) \succeq \zero.$$
The second-order stationary condition ensures CR to escape saddle points wherever the corresponding Hessian matrix has a negative eigenvalue (i.e., strict saddle points \cite{Sun2015}). In particular, many nonconvex machine learning problems have been shown to exclude spurious local minima and have only strict saddle points other than global minima \cite{Baldi1989,Sun2017,RongGe2016}. In such a desirable case, CR is guaranteed to find the global minimum for these nonconvex problems. 
To be specific, given a proper parameter $M>0$, CR generates a variable sequence $\{ \xb_{k}\}_k$ via the following update rule
\begin{align}
	\text{(CR)}~ \xb_{k+1} \in \argmin_{\yb \in \RR^d} ~\inner{\yb - \xb_{k}}{\nabla f (\xb_{k})} + \frac{1}{2} (\yb - \xb_{k})^\ttop \nabla^2 f(\xb_{k})  (\yb - \xb_{k}) + \frac{M}{6} \|\yb - \xb_{k}\|^3. \label{eq: CR}
\end{align}
Intuitively, CR updates the variable by minimizing an approximation of the objective function at the current iterate. This approximation is essentially the third-order Taylor's expansion of the objective function. We note that computationally efficient solver has been proposed for solving the above cubic subproblem \cite{Agarwal2017}, and it was shown that the resulting computation complexity of CR to achieve a second-order stationary point serves as the state-of-the-art among existing algorithms that achieve the second-order stationary condition.
%is lower than that of gradient descent to achieve a first-order stationary point.

Several studies have explored the convergence of CR for nonconvex optimization. Specifically, the pioneering work \cite{Nesterov2006} first proposed the CR algorithm and showed that the variable sequence $\{\xb_{k} \}_k$ generated by CR has second-order stationary limit points and converges sub-linearly. Furthermore, the gradient norm along the iterates was shown to converge quadratically (i.e., super-linearly) given an initialization with positive definite Hessian. Moreover, the function value along the iterates was shown to converge super-linearly under a certain gradient dominance condition of the objective function. Recently, \cite{Yue2018} studied the asymptotic convergence rate of CR under the local error bound condition, and established the quadratic convergence of the distance between the variable and the solution set. 
Clearly, these results demonstrate that special geometrical properties (such as the gradient dominance condition and the local error bound) enable much faster convergence rate of CR towards a second-order stationary point. However, these special conditions may fail to hold for generic complex objectives involved in practical applications. Thus, it is desired to further understand the convergence behaviors of CR for optimizing nonconvex functions over a broader spectrum of geometrical conditions.

The Kurdyka-{\L}ojasiewicz (K{\L}) property (see \Cref{sec: pre} for details) \cite{Bolte2007,Bolte2014} serves as such a candidate. As described in \Cref{sec: pre}, the \KL property captures a broad spectrum of the local geometries that a nonconvex function can have and is parameterized by a parameter $\theta$ that changes over its allowable range.  In fact, the \KL property has been shown to hold ubiquitously for most practical functions (see \Cref{sec: pre} for a list of examples). 
The \KL property has been exploited extensively to analyze the convergence rate of various {\em first-order} algorithms for nonconvex optimization, e.g., gradient method \cite{Attouch2009,Li2017}, alternating minimization \cite{Bolte2014} and distributed gradient methods \cite{Zhou2016}.
But it has not been exploited to establish the {\em convergence rate} of second-order algorithms towards second-order stationary points.  
In this paper, we exploit the \KL property of the objective function to provide a comprehensive study of the convergence rate of CR for nonconvex optimization. We anticipate our study to substantially advance the existing understanding of the convergence of CR to a much broader range of nonconvex functions. We summarize our contributions as follows.
\vspace{-5pt}
\subsection{Our Contributions}
\vspace{-5pt}
We characterize the convergence rate of CR locally towards second-stationary points for nonconvex optimization under the \KL property of the objective function. As the first work that establishes the convergence rate under the \KL property for a second-order algorithm, our results also compare and contrast the different order-level of guarantees that \KL yields between the first and second-order algorithms.
Specifically, we establish the convergence rate of CR in terms of the following optimality measures.

\textbf{Gradient norm \& least eigenvalue of Hessian:}
We characterize the asymptotic convergence rates of the sequence of gradient norm $\{\|\nabla f(\xb_{k})\| \}_k$ and the sequence of the least eigenvalue of Hessian $\{\lambda_{\min} (\nabla^2 f(\xb_{k}))\}_k$ generated by CR. We show that CR meets the second-order stationary condition at a (super)-linear rate in the parameter range $\theta \in [\frac{1}{3}, 1]$ for the \KL property, while the convergence rates become sub-linearly in the parameter range $\theta \in (0, \frac{1}{3})$ for the \KL property.

\textbf{Function value gap \& variable distance:} 
We characterize the asymptotic convergence rates of the function value sequence $\{f(\xb_{k})\}_k$ and the variable sequence $\{\xb_{k}\}_k$ to the function value and the variable of a second-order stationary point, respectively. The obtained convergence rates range from sub-linearly to super-linearly depending on the parameter $\theta$ associated with the \KL property. Our convergence results generalize the existing ones established in \cite{Nesterov2006} corresponding to special cases of $\theta$ to the full range of $\theta$. Furthermore, these types of convergence rates for CR are orderwise faster than the corresponding convergence rates of the first-order gradient descent algorithm in various regimes of the \KL parameter $\theta$ (see \Cref{table: 1} and \Cref{table: 2}  for the comparisons). 

Based on the \KL property, we establish a generalized local error bound condition (referred to as the \KL error bound). The \KL error bound further leads to characterization of the convergence rate of distance between the variable and the solution set, which ranges from sub-linearly to super-linearly depending on the parameter $\theta$ associated with the \KL property.
This result generalizes the existing one established in \cite{Yue2018} corresponding to a special case of $\theta$ to the full range of $\theta$. We also point out the applicability of our results to CR with inexactness in the algorithm computation.

\vspace{-5pt}
\subsection{Related Work}
\vspace{-5pt}
\textbf{Cubic regularization:} CR algorithm was first proposed in \cite{Nesterov2006}, in which the authors analyzed the convergence of CR to second-order stationary points for nonconvex optimization. In \cite{Nesterov2008}, the authors established the sub-linear convergence of CR for solving convex smooth problems, and they further proposed an accelerated version of CR with improved sub-linear convergence. Recently, \cite{Yue2018} studied the asymptotic convergence properties of CR under the error bound condition, and established the quadratic convergence of the iterates. Several other works proposed different methods to solve the cubic subproblem of CR, e.g., \cite{Agarwal2017,Carmon2016,Cartis2011a}.

\textbf{Inexact-cubic regularization:} 
Another line of work aimed at improving the computation efficiency of CR by solving the cubic subproblem with inexact gradient and Hessian information. In particular, \cite{Saeed2017,Jin2017} proposed an inexact CR for solving convex problem and nonconvex problems, respectively. Also, \cite{Cartis2011b} proposed an adaptive inexact CR for nonconvex optimization, whereas \cite{Jiang2017} further studied the accelerated version for convex optimization. Several studies explored subsampling schemes to implement inexact CR algorithms, e.g., \cite{kohler2017,Xu2017,Zhou2018,Wang2018}.

\textbf{\KL property:} The \KL property was first established in \cite{Bolte2007}, and was then widely applied to characterize the asymptotic convergence behavior for various first-order algorithms for nonconvex optimization \cite{Attouch2009,Li2017,Bolte2014,Zhou2016,Zhou_2017a}. The \KL property was also exploited to study the convergence of second-order algorithms such as generalized Newton's method \cite{Frankel2015} and the trust region method \cite{Noll2013}. However, these studies did not characterize the convergence rate and the studied methods cannot guarantee to converge to second-order stationary points, whereas this paper provides this type of results.

\textbf{First-order algorithms that escape saddle points:} Many first-order algorithms are proved to achieve second-order stationary points for nonconvex optimization. For example, online stochastic gradient descent \cite{Ge2015}, perturbed gradient descent \cite{Jin2017}, gradient descent with negative curvature \cite{Carmon2016,Liu2017} and other stochastic algorithms \cite{Zhu2017}.

\section{Preliminaries on \KL Property and CR Algorithm}\label{sec: pre}
Throughout the paper, we make the following standard assumptions on the objective function $f: \RR^d \to \RR$ \cite{Nesterov2006,Yue2018}.
\begin{assum}\label{assum: f}
	The objective function $f$ in the problem $\mathrm{(P)}$ satisfies:
	\begin{enumerate}[leftmargin=*,topsep=0pt,noitemsep]
		\item Function $f$ is continuously twice-differentiable and bounded below, i.e.,  $\inf_{\xb\in \RR^d} f(\xb) > -\infty$;
		\item For any $\alpha \in \RR$, the sub-level set $\{\xb: f(\xb) \le \alpha\}$ is compact; 
		\item The Hessian of $f$ is $L$-Lipschitz continuous on a compact set $\mathcal{C}$, i.e., 
		\begin{align*}
			\|\nabla^2 f(\xb) - \nabla^2f (\yb)\| \le L \|\xb - \yb\|, \quad \xb, \yb \in \mathcal{C}.
		\end{align*}
	\end{enumerate}
\end{assum} 

%Note that as the cubic method reduces the level set, we only need $\nabla^2 f(\xb)$ to be Lipschitz in the compact set $\{\xb: f(\xb) \le f(\xb_0)\}$.

The above assumptions make problem (P) have a solution and make the iterative rule of CR being well defined. Besides these assumptions, many practical functions are shown to satisfy the so-called {\L}ojasiewicz gradient inequality \cite{Lojas1965}. Recently, such a condition was further generalized to the Kurdyka-{\L}ojasiewicz property \cite{Bolte2007,Bolte2014}, which is satisfied by a larger class of nonconvex functions. 

Next, we introduce the \KL property of a function $f$.
Throughout, the (limiting) subdifferential of a proper and lower-semicontinuous function $f$ is denoted as $\partial f$, and the point-to-set distance is denoted as $\dist_\Omega(\xb) := \inf_{\wb\in \Omega} \|\xb - \wb\|$.
 
\begin{definition}[\KL property, \cite{Bolte2014}]\label{def: KL}
	A proper and lower-semicontinuous function $f$ is said to satisfy the \KL property if for every compact set $\Omega\subset \mathrm{dom}f$ on which $f$ takes a constant value $f_\Omega \in \RR$, there exist $\varepsilon, \lambda >0$ such that for all $\bar{\xb} \in \Omega$ and all $\xb\in \{\zb\in \RR^d : \dist_\Omega(\zb)<\varepsilon, f_\Omega < f(\zb) <f_\Omega + \lambda\}$, one has
	\begin{align}\label{eq: KL}
	\varphi' \left(f(\xb) - f_\Omega\right) \cdot \dist_{\partial f(\xb)}(\zero) \ge 1,
	\end{align}
	where function $\varphi: [0,\lambda) \to \RR_+$ takes the form $\varphi(t) = \frac{c}{\theta} t^\theta$ for some constants $c>0, \theta\in (0,1]$.
\end{definition}
The \KL property establishes local geometry of the nonconvex function around a compact set. In particular, consider a differentiable function (which is our interest here) so that $\partial f = \nabla f$. Then, the local geometry described in \cref{eq: KL} can be rewritten as 
\begin{align}
	f(\xb) - f_\Omega \le C \|\nabla f(\xb)\|^{\frac{1}{1-\theta}} \label{eq: KLsimple}
\end{align}
for some constant $C>0$. \Cref{eq: KLsimple} can be viewed as a generalization of the well known gradient dominance condition \cite{Lojasiewicz1963,Karimi2016}, which corresponds to the special case of $\theta = \frac{1}{2}$ and is satisfied by various nonconvex machine learning models \cite{Zhou2016b,Yue2018,Zhou2017}. In general, the parameter $\theta \in (0,1]$ of the \KL property captures the local curvature of the function, and we show in our main results that it determines the asymptotic convergence rate of CR.

The \KL property has been shown to hold for a large class of functions including sub-analytic functions, logarithm and exponential functions and semi-algebraic functions. These function classes cover most of nonconvex objective functions encountered in practical applications. We provide a partial list of \KL functions below and refer the reader to \cite{Bolte2014,Attouch2010} for a comprehensive list.
\begin{itemize}[leftmargin=*,topsep=0pt,noitemsep]
	\item Real polynomial functions;
	\item Vector (semi)-norms $\|\cdot\|_p$ with $p\ge 0$ be any rational number;
	\item Matrix (semi)-norms, e.g., operator norm, nuclear norm, Frobenious norm, rank, etc;
	\item Logarithm functions, exponentiation functions;
\end{itemize}

Next, we provide some fundamental understandings of CR that determines the convergence rate. The algorithmic dynamics of CR \cite{Nesterov2006} is very different from that of first-order gradient descent algorithm, which implies that the convergence rate of CR under the \KL property can be very different from the existing result for the first-order algorithms under the \KL property. We provide a detailed comparison between the two algorithmic dynamics in \Cref{table: 1} for illustration, where $L_{\textrm{grad}}$ corresponds to the Lipschitz parameter of $\nabla f$ for gradient descent, $L$ is the Lipschitz parameter for Hessian  and we choose $M=L$ for CR for simple illustration.

\begin{table}[ht]
	\caption{Comparison between dynamics of gradient descent and CR.}\label{table: 1}
	\center
	\begin{tabular}{cccc}
		\toprule
		{} & {} & {gradient descent} & {cubic-regularization $(M=L)$} \\ \midrule
		$f(\xb_{k}) - f(\xb_{k-1})$  & $\le$ & $-\tfrac{L_{\textrm{grad}}}{2}\|\xb_{k} - \xb_{k-1}\|^2$ & $- \frac{L}{12} \|\xb_{k} - \xb_{k-1}\|^3$  \\
		\midrule
		$\|\nabla f(\xb_{k})\|$  & $\le$   & $L_{\textrm{grad}}\|\xb_{k} - \xb_{k-1}\|$ & $ L\|\xb_{k} - \xb_{k-1}\|^2$   \\
		\midrule
		$-\lambda_{\min}(\nabla^2 f(\xb_{k}))$  & $\le$  & N/A & $\tfrac{3L }{2} \|\xb_{k} - \xb_{k-1}\|$  \\
		\bottomrule
	\end{tabular}
\end{table}
It can be seen from \Cref{table: 1} that the dynamics of gradient descent involves information up to the first order (i.e., function value and gradient), whereas the dynamics of CR involves the additional second order information, i.e., least eigenvalue of Hessian (last line in \Cref{table: 1}).
In particular, note that the successive difference of function value $f(\xb_{k+1}) - f(\xb_{k})$ and the gradient norm $\|\nabla f(\xb_{k})\|$ of CR are bounded by higher order terms of $\|\xb_{k} - \xb_{k-1}\|$ compared to those of gradient descent. Intuitively, this implies that CR should converge faster than gradient descent in the converging phase when $\xb_{k+1} - \xb_{k} \to \zero$. Next, we exploit the dynamics of CR and the \KL property  to study its asymptotic convergence rate.

\textbf{Notation:} Throughout the paper, we denote $f(n) = \Theta(g(n))$ if and only if for some $0<c_1<c_2$, $c_1 g(n) \le f(n) \le c_2 g(n)$ for all $n\ge n_0$.

\section{Convergence Rate of CR to Second-order Stationary Condition}
In this subsection, we explore the convergence rates of the gradient norm and the least eigenvalue of the Hessian along the iterates generated by CR under the \KL property.
Define the second-order stationary gap
\begin{align*}
\mu(\xb) := \max \bigg\{\sqrt{\frac{2}{L+M} \|\nabla f(\xb)\|},~ -\frac{2}{2L+M} \lambda_{\min}(\nabla^2 f(\xb)) \bigg\}.
\end{align*}
The above quantity is well established as a criterion for achieving second-order stationary \cite{Nesterov2006}. 
It tracks both the gradient norm and the least eigenvalue of the Hessian at $\xb$. In particular, the second-order stationary condition is satisfied as $\mu(\xb) = 0$. 

Next, we characterize the convergence rate of $\mu$ for CR under the \KL property.
\begin{restatable}{thm}{thmmu}\label{coro: mu}
	Let \Cref{assum: f} hold and assume that problem $\mathrm{(P)}$ satisfies the \KL property associated with parameter $\theta \in (0,1]$. Then, there exists a sufficiently large $k_0\in \mathds{N}$ such that for all $k\ge k_0$ the sequence $\{\mu(\xb_{k})\}_k$ generated by CR satisfies
	\begin{enumerate}[leftmargin=*,topsep=0pt,noitemsep]
		\item If $\theta = 1$, then $\mu(\xb_{k}) \to 0$ within finite number of iterations; 
		\item If $\theta \in (\frac{1}{3}, 1)$, then $\mu(\xb_{k})\to 0$ super-linearly as 
		$\mu(\xb_{k}) \le \Theta \Big( \exp \Big(-\big(\frac{2\theta}{1-\theta}\big)^{k-k_0}\Big) \Big);$
		\item If $\theta = \frac{1}{3}$, then $\mu(\xb_{k})\to 0$ linearly as 
		$\mu(\xb_{k}) \le \Theta \Big( \exp \big(-(k-k_0)\big) \Big);$
		\item If $\theta \in (0, \frac{1}{3})$, then $\mu(\xb_{k}) \to 0$ sub-linearly as 
		$\mu(\xb_{k}) \le \Theta \Big((k-k_0)^{-\frac{2\theta}{1-3\theta}}\Big).$
	\end{enumerate}
\end{restatable}
\Cref{coro: mu} provides a full characterization of the convergence rate of $\mu$ for CR to meet the second-order stationary condition under the \KL property. It can be seen from \Cref{coro: mu} that the convergence rate of $\mu$ is determined by the \KL parameter $\theta$. Intuitively, in the regime $\theta\in (0,\frac{1}{3}]$ where the local geometry is `flat', CR achieves second-order stationary slowly as (sub)-linearly. As a comparison, in the regime $\theta \in (\frac{1}{3}, 1]$ where the local geometry is `sharp', CR achieves second-order stationary fast as super-linearly.

We next compare the convergence results of $\mu$ in \Cref{coro: mu} with that of $\mu$ for CR studied in \cite{Nesterov2006}. To be specific, the following two results are established in \cite[Theorems 1 \& 3]{Nesterov2006} under \Cref{assum: f}.
\begin{enumerate}[topsep=0pt]
	\item $\lim_{k\to \infty} \mu(\xb_{k}) = 0, \quad \min_{1\le t \le k} \mu(\xb_t) \le \Theta (k^{-\frac{1}{3}})$;
	\item If the Hessian is positive definite at certain $\xb_{t}$, then the Hessian remains to be positive definite at all subsequent iterates $\{\xb_{k}\}_{k\ge t}$ and $\{\mu(\xb_{k})\}_k$ converges to zero quadratically.
\end{enumerate}
Item 1 establishes a best-case bound for $\mu(\xb_{k})$, i.e., it holds only for the minimum $\mu$ along the iteration path. As a comparison, our results in \Cref{coro: mu} characterize the convergence rate of $\mu(\xb_{k})$ along the entire asymptotic iteration path. Also, the quadratic convergence in item 2 relies on the fact that CR eventually stays in a locally strong convex region (with \KL parameter $\theta = \frac{1}{2}$), and this is consistent with our convergence rate in \Cref{coro: mu} for \KL parameter $\theta = \frac{1}{2}$. In summary, our result captures the effect of the underlying geometry (parameterized by the \KL parameter $\theta$) on the convergence rate of CR towards second-order stationary.

%In \cite{Nesterov2006}, the following result regarding $\{\mu(\xb_k)\}_k$ generated by CR was established. 
%\begin{thm}\label{thm: nes_eigen}
%	Let \Cref{assum: f} hold. Then, the sequence $\{\xb_{k}\}_k$ generated by CR satisfies
%	Moreover,  Then, the following properties hold.
%	\begin{enumerate}[leftmargin=*,topsep=0pt,noitemsep]
%		\item For all $k\ge t$, $\lambda_{\min}(\nabla^2 f(\xb_k)) > 0$ and stays in a bounded interval;
%		\item For all $k\ge t$, the sequence 
%	\end{enumerate}
%\end{thm}

%Thus, if the Hessian of the variable sequence generated by CR fails to meet the second-order stationary condition, its least eigenvalue converges to zero at the corresponding rates parameterized by the \KL parameter $\theta$ in \Cref{thm: eigen}. Now, we can further combine the results in Theorems \ref{thm: nes_eigen}, \ref{thm: grad} and \ref{thm: eigen} to obtain a full characterization of the convergence of $\mu(\xb_{k})$ generated by CR under the \KL property.

\section{Other Convergence Results of CR}
In this section, we first present the convergence rate of function value and variable distance for CR under the \KL property.
Then, we discuss that such convergence results are also applicable to characterize the convergence rate of inexact CR.
\vspace{-5pt}
\subsection{Convergence Rate of Function Value for CR}
\vspace{-5pt}
%We first recall the following fundamental result proved in \cite{Nesterov2006}, which serves as a convenient reference.
%\begin{thm}[Theorem 2, \cite{Nesterov2006}]\label{thm: 1}
%	Let \Cref{assum: f} hold. Then, the sequence $\{\xb_{k}\}_k$ generated by CR satisfies 
%	\begin{enumerate}[leftmargin=*,topsep=0pt,noitemsep]
%		\item The set of limit points $\omega(\xb_{0})$ of $\{\xb_k\}_k$ is nonempty and compact, all of which are second-order stationary points;
%		\item The sequence $\{f(\xb_{k}) \}_k$ decreases to a finite limit $\bar{f}$, which is the constant function value evaluated on the set $\omega(\xb_{0})$. 
%	\end{enumerate}
%\end{thm}
%Clearly, the above properties of set $\omega(\xb_{0})$ meet all assumptions of \KL property.

It has been proved in \cite[Theorem 2]{Nesterov2006} that the function value sequence $\{f(\xb_{k})\}_k$ generated by CR decreases to a finite limit $\bar{f}$, which corresponds to the function value evaluated at a certain second-order stationary point. The corresponding convergence rate has also been developed in \cite{Nesterov2006} under certain gradient dominance condition of the objective function. 
In this section, we characterize the convergence rate of $\{f(\xb_{k}) \}_k$ to $\bar{f}$ for CR by exploiting the more general \KL property. We obtain the following result.
\begin{restatable}{thm}{thmfunc}\label{thm: 2}
	Let \Cref{assum: f} hold and assume problem $\mathrm{(P)}$ satisfies the \KL property associated with parameter $\theta \in (0,1]$. Then, there exists a sufficiently large $k_0\in \mathds{N}$ such that for all $k\ge k_0$ the sequence $\{f(\xb_k) \}_k$ generated by CR satisfies
	\begin{enumerate}[leftmargin=*,topsep=0pt,noitemsep]
		\item If $\theta = 1$, then $f(\xb_{k})\downarrow \bar{f}$ within finite number of iterations;
		\item If $\theta \in (\frac{1}{3}, 1)$, then $f(\xb_{k})\downarrow \bar{f}$ super-linearly as 
		$f(\xb_{k+1}) - \bar{f} \le \Theta \Big( \exp \Big(-\big(\frac{2}{3(1-\theta)}\big)^{k-k_0}\Big) \Big);$
		\item If $\theta = \frac{1}{3}$, then $f(\xb_{k})\downarrow \bar{f}$ linearly as 
		$f(\xb_{k+1}) - \bar{f} \le \Theta \Big( \exp \big(-(k-k_0)\big) \Big);$
		\item If $\theta \in (0, \frac{1}{3})$, then $f(\xb_{k})\downarrow \bar{f}$ sub-linearly as 
		$f(\xb_{k+1}) - \bar{f} \le \Theta \Big((k-k_0)^{-\frac{2}{1-3\theta}}\Big).$
	\end{enumerate}
\end{restatable}
From \Cref{thm: 2}, it can be seen that the function value sequence generated by CR has diverse asymptotic convergence rates in different regimes of the \KL parameter $\theta$. In particular, a larger $\theta$ implies a sharper local geometry that further facilitates the convergence. We note that the gradient dominance condition
discussed in \cite{Nesterov2006} locally corresponds to the \KL property in \cref{eq: KLsimple} with the special cases $\theta \in \{0, \frac{1}{2}\}$, and hence the convergence rate results in \Cref{thm: 2} generalize those in \cite[Theorems 6, 7]{Nesterov2006}. 

We can further compare the function value convergence rates of CR with those of gradient descent method \cite{Frankel2015} under the \KL property (see \Cref{table: 2} for the comparison).
In the \KL parameter regime $\theta \in [\frac{1}{3}, 1]$, the convergence rate of $\{f(\xb_{k}) \}_k$ for CR is super-linear---orderwise faster than the corresponding (sub)-linear convergence rate of gradient descent. Also, both methods converge sub-linearly in the parameter regime $\theta \in (0, \frac{1}{3})$, and the corresponding convergence rate of CR is still considerably faster than that of gradient descent. 
\begin{table}[ht]
	\caption{Comparison of convergence rate of $\{f(\xb_{k}) \}_k$ between gradient descent and CR.}\label{table: 2}
	\center
	\begin{tabular}{ccc}
		\toprule
		{\KL parameter} & {gradient descent} & {cubic-regularization} \\ \midrule
		$\theta = 1$   & finite-step & finite-step  \\
		\midrule
		$\theta \in [\frac{1}{2}, 1)$    & linear & super-linear  \\
		\midrule
		$\theta \in [\frac{1}{3}, \frac{1}{2})$   & sub-linear & (super)-linear \\
		\midrule
		$\theta \in (0, \frac{1}{3})$   & sub-linear $\mathcal{O}(k^{-\frac{1}{1-2\theta}})$ & sub-linear $\mathcal{O}(k^{-\frac{2}{1-3\theta}})$\\
		\bottomrule
	\end{tabular}
\end{table}
\vspace{-5pt}
\subsection{Convergence Rate of Variable Distance for CR}
\vspace{-5pt}
It has been proved in \cite[Theorem 2]{Nesterov2006} that all limit points of $\{\xb_{k}\}_k$ generated by CR are second order stationary points. However, the sequence is not guaranteed to be convergent and no convergence rate is established.

Our next two results show that the sequence $\{\xb_{k}\}_k$ generated by CR is convergent under the \KL property.
\begin{restatable}{thm}{thmfinitelen}\label{thm: finite length}
	Let \Cref{assum: f} hold and assume that problem $\mathrm{(P)}$ satisfies the \KL property. Then, the sequence $\{\xb_k \}_k$ generated by CR satisfies
	\begin{align}
	\sum_{k=0}^{\infty} \|\xb_{k+1} - \xb_{k} \| < +\infty. \label{eq: 4}
	\end{align}
\end{restatable}
\Cref{thm: finite length} implies that CR generates an iterate $\{\xb_{k} \}_k$ with finite trajectory length, i.e., $\|\xb_{\infty} - \xb_{0}\|< +\infty$.
In particular, \cref{eq: 4} shows that the sequence is absolutely summable, which strengthens the result in \cite{Nesterov2006} that establishes the cubic summability instead, i.e., $\sum_{k=0}^{\infty} \|\xb_{k+1} - \xb_{k} \|^3 < +\infty$. We note that the cubic summability property does not guarantee that the sequence is convergence, for example, the sequence $\{\frac{1}{k} \}_k$ is cubic summable but is not absolutely summable. In comparison, the summability property in \cref{eq: 4} directly implies that the sequence $\{\xb_k \}_k$ generated by CR is a Cauchy convergent sequence, and we obtain the following result.
\begin{restatable}{coro}{corofinitelen}\label{coro: finite length}
	Let \Cref{assum: f} hold and assume that problem $\mathrm{(P)}$ satisfies the \KL property. Then, the sequence $\{\xb_k \}_k$ generated by CR is a Cauchy sequence and converges to some second-order stationary point $\bar{\xb}$.
\end{restatable}

We note that the convergence of $\{\xb_{k} \}_k$ to a second-order stationary point is also established for CR in \cite{Yue2018}, but under the special error bound condition, whereas we establish the convergence of $\{\xb_{k} \}_k$ under the \KL property that holds for general nonconvex functions.

Next, we establish the convergence rate of $\{\xb_{k} \}_k$ to the second-order stationary limit $\bar{\xb}$.
\begin{restatable}{thm}{thmiterate}\label{thm: converge_ite}
	Let \Cref{assum: f} hold and assume that problem $\mathrm{(P)}$ satisfies the \KL property. Then, there exists a sufficiently large $k_0\in \mathds{N}$ such that for all $k\ge k_0$ the sequence $\{\xb_k\}_k$ generated by CR satisfies
	\begin{enumerate}[leftmargin=*,topsep=0pt,noitemsep]
		\item If $\theta = 1$, then $\xb_{k}\to \bar{\xb}$ within finite number of iterations;
		\item If $\theta \in (\frac{1}{3}, 1)$, then $\xb_{k}\to \bar{\xb}$ super-linearly as 
		$\|\xb_{k+1} - \bar{\xb}\| \le \Theta \Big( \exp \Big(-\big(\frac{2\theta}{3(1-\theta)} + \frac{2}{3}\big)^{k-k_0}\Big) \Big);$
		\item If $\theta = \frac{1}{3}$, then $\xb_{k}\to \bar{\xb}$ linearly as 
		$\|\xb_{k+1} - \bar{\xb}\| \le \Theta \Big( \exp \big(-(k-k_0)\big) \Big);$
		\item If $\theta \in (0, \frac{1}{3})$, then $\xb_{k}\to \bar{\xb}$ sub-linearly as 
		$\|\xb_{k+1} - \bar{\xb}\| \le \Theta \Big((k-k_0)^{-\frac{2\theta}{1-3\theta}}\Big).$
	\end{enumerate}
\end{restatable}
From \Cref{thm: converge_ite}, it can be seen that the convergence rate of $\{\xb_{k}\}_k$ is similar to that of $\{f(\xb_{k})\}_k$ in \Cref{thm: 2} in the corresponding regimes of the \KL parameter $\theta$. Essentially, a larger parameter $\theta$ induces a sharper local geometry that leads to a faster convergence. 

We can further compare the variable convergence rate of CR in \Cref{thm: converge_ite} with that of gradient descent method \cite{Attouch2009} (see \Cref{table: 3} for the comparison).  It can be seen that the variable sequence generated by CR converges orderwise faster than that generated by the gradient descent method in a large parameter regimes of $\theta$. 

\begin{table}[ht]
	\caption{Comparison of convergence rate of $\{\xb_{k} \}_k$ between gradient descent and CR.}\label{table: 3}
	\center
	\begin{tabular}{ccc}
		\toprule
		{\KL parameter} & {gradient descent} & {cubic-regularization} \\ \midrule
		$\theta = 1$   & finite-step & finite-step  \\
		\midrule
		$\theta \in [\frac{1}{2}, 1)$    & linear & super-linear  \\
		\midrule
		$\theta \in [\frac{1}{3}, \frac{1}{2})$   & sub-linear & (super)-linear \\
		\midrule
		$\theta \in (0, \frac{1}{3})$   & sub-linear $\mathcal{O}(k^{-\frac{\theta}{1-2\theta}})$ & sub-linear $\mathcal{O}(k^{-\frac{2\theta}{1-3\theta}})$\\
		\bottomrule
	\end{tabular}
\end{table}
\vspace{-5pt}
\subsection{Extension to Inexact Cubic Regularization}
\vspace{-5pt}
All our convergence results for CR in previous sections are based on the algorithm dynamics \Cref{table: 1} and the \KL property of the objective function. In fact, such dynamics of CR has been shown to be satisfied by other inexact variants of CR \cite{Cartis2011a,Cartis2011b,kohler2017,Wang2018} with different constant terms. These inexact variants of CR updates the variable by solving the cubic subproblem in \cref{eq: CR} with the inexact gradient $\nabla \widehat{f}(\xb_{k})$ and inexact Hessian
$\nabla^2 \widehat{f}(\xb_{k})$ that satisfy the following inexact criterion
\begin{align}
	&\|\nabla \widehat{f}(\xb_{k}) - {\nabla f}(\xb_{k})\| \le c_1 \|\xb_{k+1} - \xb_{k}\|^2, \\
	&\big\|\big(\nabla^2 \widehat{f}(\xb_{k}) - {\nabla^2 f}(\xb_{k})\big)(\xb_{k+1} - \xb_{k}) \big\| \le c_2 \|\xb_{k+1} - \xb_{k}\|^2,
\end{align}
where $c_1, c_2$ are positive constants. Such inexact criterion can reduce the computational complexity of CR and can be realized via various types of subsampling schemes \cite{kohler2017,Wang2018}.

Since the above inexact-CR also satisfies the dynamics in \Cref{table: 1} (with different constant terms), all our convergence results for CR can be directly applied to inexact CR. Then, we obtain the following corollary.    
\begin{coro}[Inexact-CR]
	Let \Cref{assum: f} hold and assume that problem $\mathrm{(P)}$ satisfies the \KL property. Then, the sequences $\{\mu(\xb_{k}) \}_k, \{f(\xb_{k}) \}_k, \{\xb_{k}\}_k$ generated by the inexact-CR satisfy respectively the results in Theorems \ref{coro: mu}, \ref{thm: 2}, \ref{thm: finite length} and \ref{thm: converge_ite}.
\end{coro}

\section{Convergence Rate of CR under \KL Error Bound}
In \cite{Yue2018}, it was shown that the gradient dominance condition (i.e., \cref{eq: KLsimple} with $\theta = \frac{1}{2}$) implies the following local error bound, which further leads to the quadratic convergence of CR. 
\begin{definition}
	Denote $\Omega$ as the set of second-order stationary points of $f$. Then, $f$ is said to satisfy the local error bound condition if there exists $\kappa, \rho >0$ such that
	\begin{align}
		\dist_\Omega (\xb) \le \kappa \|\nabla f(\xb)\|, \quad \forall~ \dist_\Omega (\xb) \le \rho. \label{eq: EB}
	\end{align}
\end{definition}
As the \KL property generalizes the gradient dominance condition, it implies a much more general spectrum of the geometry that includes the error bound in \cite{Yue2018} as a special case. Next, we first show that the \KL property implies the \KL-error bound, and then exploit such an error bound to establish the convergence of CR. 
\begin{restatable}{proposition}{propeb}\label{prop: 1}
	Denote $\Omega$ as the set of second-order stationary points of $f$. Let \Cref{assum: f} hold and assume that $f$ satisfies the \KL property. Then, there exist $\kappa, \varepsilon, \lambda >0$  such that for all $\xb\in \{\zb\in \RR^d : \dist_\Omega(\zb)<\varepsilon, f_\Omega < f(\zb) <f_\Omega + \lambda\}$, the following property holds.
	\begin{align}
	%(\text{\KL growth condition})\quad&f(\xb) - f_\Omega \ge C [\dist_\Omega (\xb)]^{\frac{1}{\theta}}, \label{eq: grow}\\
	(\text{\KL-error bound})\quad&\dist_\Omega (\xb) \le \kappa \|\nabla f(\xb)\|^{\frac{\theta}{1-\theta}}. \label{eq: eb}
	\end{align}
\end{restatable}	
We refer to the condition in \cref{eq: eb} as the  \KL-error bound, which generalizes the original error bound in \cref{eq: EB} under the \KL property. In particular, the \KL-error bound reduces to the error bound in the special case $\theta = \frac{1}{2}$. 
%Intuitively, the \KL growth condition in \cref{eq: grow} characterizes how fast the function value grows when $\xb$ deviates the set $\Omega$, and a larger $\theta$ corresponds to a faster growth of the function value. On the other hand, the \KL error bound condition in \cref{eq: eb} bounds the variable distance to set $\Omega$ in terms of the gradient norm.
By exploiting the \KL error bound, we obtain the following convergence result regarding $\dist_\Omega(\xb_{k})$.

\begin{restatable}{proposition}{thmdist}\label{thm: dist}
	Denote $\Omega$ as the set of second-order stationary points of $f$. Let \Cref{assum: f} hold and assume that problem $\mathrm{(P)}$ satisfies the \KL property. Then, there exists a sufficiently large $k_0\in \mathds{N}$ such that for all $k\ge k_0$ the sequence $\{\dist_\Omega (\xb_{k}) \}_k$ generated by CR satisfies 
	\begin{enumerate}[leftmargin=*,topsep=0pt,noitemsep]
		\item If $\theta = 1$, then $\dist_\Omega(\xb_k) \to 0$ within finite number of iterations;
		\item If $\theta \in (\frac{1}{3}, 1)$, then $\dist_\Omega(\xb_k) \to 0$ super-linearly as 
		$\dist_\Omega(\xb_{k}) \le \Theta \Big( \exp \Big(-\big(\frac{2\theta}{1-\theta}\big)^{k-k_0}\Big) \Big);$
		\item If $\theta = \frac{1}{3}$, then $\dist_\Omega(\xb_k) \to 0$ linearly as 
		$\dist_\Omega(\xb_k) \le \Theta \Big( \exp \big(-(k-k_0)\big) \Big); $
		\item If $\theta \in (0, \frac{1}{3})$, then $\dist_\Omega(\xb_k) \to 0$ sub-linearly as 
		$\dist_\Omega(\xb_k) \le \Theta \Big((k-k_0)^{-\frac{2\theta}{1-3\theta}}\Big).$
	\end{enumerate}
\end{restatable}
We note that \Cref{thm: dist} characterizes the convergence rate of the point-to-set distance $\dist_\Omega(\xb_{k})$, which is different from the convergence rate of the point-to-point distance $\|\xb_{k} - \bar{\xb}\|$ established in \Cref{thm: converge_ite}. Also, the convergence rate results in \Cref{thm: dist} generalizes the quadratic convergence result in \cite{Yue2018} that corresponds to the case $\theta = \frac{1}{2}$.

\section{Conclusion}
In this paper, we explore the asymptotic convergence rates of the CR algorithm under the \KL property of the nonconvex objective function, and establish the convergence rates of function value gap, iterate distance and second-order stationary gap for CR. Our results show that the convergence behavior of CR ranges from sub-linear convergence to super-linear convergence depending on the parameter of the  underlying  \KL geometry, and the obtained convergence rates are order-wise improved compared to those of first-order algorithms under the \KL property. As a future direction, it is interesting to study the convergence of other computationally efficient variants of the CR algorithm such as stochastic variance-reduced CR under the \KL property in nonconvex optimization.

{%\small
	\bibliographystyle{apalike}
	\bibliography{./ref}
}

\clearpage
\appendix{
	{\centering\Large \textbf{Supplementary Materials}}

The proof of \Cref{coro: mu} is based on the results in other theorems. Thus, we postpone its proof to the end of the supplementary material.

\section*{Proof of \Cref{thm: 2}}

\thmfunc*
\begin{proof}
	We first recall the following fundamental result proved in \cite{Nesterov2006}, which serves as a convenient reference.
	\begin{thm}[Theorem 2, \cite{Nesterov2006}]\label{thm: 1}
		Let \Cref{assum: f} hold. Then, the sequence $\{\xb_{k}\}_k$ generated by CR satisfies 
		\begin{enumerate}[leftmargin=*,topsep=0pt,noitemsep]
			\item The set of limit points $\omega(\xb_{0})$ of $\{\xb_k\}_k$ is nonempty and compact, all of which are second-order stationary points;
			\item The sequence $\{f(\xb_{k}) \}_k$ decreases to a finite limit $\bar{f}$, which is the constant function value evaluated on the set $\omega(\xb_{0})$. 
		\end{enumerate}
	\end{thm}
	
	From the results of \Cref{thm: 1} we conclude that $\dist_{\omega(\xb_{0})}(\xb_{k}) \to 0$, $f(\xb_{k}) \downarrow \bar{f}$ and $\omega(\xb_{0})$ is a compact set on which the function value is the constant $\bar{f}$. Then, it is clear that for any fixed $\epsilon>0, \lambda>0$ and all $k\ge k_0$ with $k_0$ being sufficiently large, $\xb_k \in \{\xb: \dist_{\omega(\xb_{0})}(\xb)<\varepsilon, \bar{f} < f(\xb) <\bar{f} + \lambda\}$. Hence, all the conditions of the \KL property in \Cref{def: KL} are satisfied, and we can exploit the \KL inequality in \cref{eq: KL}. 
	
	Denote $r_k := f(\xb_{k}) - \bar{f}$. For all $k\ge k_0$ we obtain that 
	\begin{align}
	r_k \overset{(i)}{\le} C \|\nabla f(\xb_{k})\|^{\frac{1}{1-\theta}} \overset{(ii)}{\le} C\|\xb_{k} - \xb_{k-1}\|^{\frac{2}{1-\theta}} \overset{(iii)}{\le} C(r_{k-1} - r_{k})^{\frac{2}{3(1-\theta)}}, \label{eq: supp3}
	\end{align}
	where (i) follows from the \KL property in \cref{eq: KLsimple}, (ii) and (iii) follow from the dynamics of CR in \Cref{table: 1} and we have absorbed all constants into $C$. Define $\delta_k = r_k C^{\frac{3(1-\theta)}{3\theta-1}}$, then the above inequality can be rewritten as
	\begin{align}
	\delta_{k-1} - \delta_{k} \ge \delta_k^{\frac{3(1-\theta)}{2}}, \quad \forall k \ge k_0. \label{eq: supp2}
	\end{align}
	Next, we discuss the convergence rate of $\delta_k$ under different regimes of $\theta$. 
	
	\textbf{Case 1: $\theta = 1$.} 
	
	In this case, the \KL property in \cref{eq: KL} satisfies $\varphi'(t) = c$ and implies that $\|\nabla f(\xb_{k})\| \ge \frac{1}{c}$ for some constant $c>0$. On the other hand, by the dynamics of CR in \Cref{table: 1}, we obtain that 
	\begin{align}
	f(\xb_{k+1}) \le f(\xb_{k}) - \frac{M}{12} \|\xb_{k+1} - \xb_{k} \|^3 \le f(\xb_{k}) - \frac{M}{12} (\frac{2}{L+M})^{\frac{3}{2}}\|\nabla f(\xb_{k})\|^{\frac{3}{2}}.
	\end{align}
	Combining these two facts yields the conclusion that for all $k\ge k_0$
	$$f(\xb_{k+1}) \le f(\xb_{k}) - C$$
	for some constant $C>0$. Then, we conclude that $f(\xb_{k}) \downarrow -\infty$, which contradicts the fact that $f(\xb_{k}) \downarrow \bar{f} > -\infty$ (since $f$ is bounded below). Hence, we must have $f(\xb_{k}) \equiv \bar{f}$ for all sufficiently large $k$. 
	
	\textbf{Case 2: $\theta \in (\frac{1}{3}, 1)$.} 
	
	In this case $0< \frac{3(1-\theta)}{2} <1$. Since $\delta_k \to 0$ as $r_k \to 0$, $\delta_k^{\frac{3(1-\theta)}{2}}$ is order-wise larger than $\delta_k$ for all sufficiently large $k$. Hence, for all sufficiently large $k$, \cref{eq: supp2} reduces to
	\begin{align}
	\delta_{k-1} \ge \delta_k^{\frac{3(1-\theta)}{2}}.
	\end{align}
	It follows that $\delta_k\downarrow 0$ super-linearly as $\delta_k \le \delta_{k-1}^{\frac{2}{3(1-\theta)}}$. Since $\delta_k = r_k C^{\frac{2}{1-3\theta}}$, we conclude that $r_k \downarrow 0$ super-linearly as $r_k \le C_1r_{k-1}^{\frac{2}{3(1-\theta)}}$ for some constant $C_1>0$. By letting $k_0$ be sufficiently large so that $r_{k_0}$ is sufficiently small, we obtain that
	\begin{align}
	r_k \le C_1r_{k-1}^{\frac{2}{3(1-\theta)}} \le C_1^{k-k_0} r_{k_0}^{(\frac{2}{3(1-\theta)})^{k-k_0}} = \Theta \Bigg( \exp \bigg(-\bigg(\frac{2}{3(1-\theta)}\bigg)^{k-k_0}\bigg) \Bigg).
	\end{align}
	
	\textbf{Case 3: $\theta = \frac{1}{3}$.} 
	
	In this case $\frac{3(1-\theta)}{2} = 1$, and \cref{eq: supp3} reduces to $r_k \le C (r_{k-1} - r_k)$, i.e., $r_k \downarrow 0$ linearly as $r_k \le \frac{C}{1+C} r_{k-1}$ for some constant $C>0$. Thus, we obtain that for all $k \ge k_0$
	\begin{align}
	r_k \le \bigg(\frac{C}{1+C}\bigg)^{k-k_0} r_{k_0} = \Theta \Big( \exp \big(-(k-k_0)\big) \Big).
	\end{align}
	
	\textbf{Case 4: $\theta \in (0, \frac{1}{3})$.} 
	
	In this case, $1< \frac{3(1-\theta)}{2} <\frac{3}{2}$ and $-\frac{1}{2} < \frac{3\theta-1}{2} < 0$. Since $\delta_k\downarrow 0$, we conclude that for all $k\ge k_0$
	\begin{align}
	\delta_{k-1}^{-\frac{3(1-\theta)}{2}} < \delta_{k}^{-\frac{3(1-\theta)}{2}}, \quad\delta_{k-1}^{\frac{3\theta-1}{2}} < \delta_{k}^{\frac{3\theta-1}{2}}.
	\end{align}
	Define an auxiliary function $\phi(t):= \frac{2}{1-3\theta} t^{\frac{3\theta-1}{2}}$ so that $\phi'(t) = -t^{\frac{3(\theta-1)}{2}}$. We next consider two cases. First, suppose that $\delta_k^{\frac{3(\theta-1)}{2}} \le 2 \delta_{k-1}^{\frac{3(\theta-1)}{2}}$. Then for all $k\ge k_0$
	\begin{align}
	\phi(\delta_k) - \phi(\delta_{k-1}) &= \int_{\delta_{k-1}}^{\delta_{k}} \phi'(t) dt =  \int_{\delta_{k}}^{\delta_{k-1}} t^{\frac{3(\theta-1)}{2}} dt \ge (\delta_{k-1} - \delta_{k}) \delta_{k-1}^{\frac{3(\theta-1)}{2}} \\
	&\overset{(i)}{\ge} \frac{1}{2}(\delta_{k-1} - \delta_{k}) \delta_{k}^{\frac{3(\theta-1)}{2}} \overset{(ii)}{\ge} \frac{1}{2},
	\end{align}
	where (i) utilizes the assumption and (ii) uses \cref{eq: supp2}. 
	
	Second, suppose that $\delta_k^{\frac{3(\theta-1)}{2}} \ge 2 \delta_{k-1}^{\frac{3(\theta-1)}{2}}$. Then $\delta_{k}^{\frac{3\theta - 1}{2}} \ge 2^{\frac{3\theta - 1}{3(\theta - 1)}} \delta_{k-1}^{\frac{3\theta - 1}{2}}$, which further leads to
	\begin{align}
	\phi(\delta_{k}) - \phi(\delta_{k-1}) &= \frac{2}{1-3\theta} (\delta_{k}^{\frac{3\theta - 1}{2}} - \delta_{k-1}^{\frac{3\theta - 1}{2}}) \ge \frac{2}{1-3\theta} (2^{\frac{3\theta - 1}{3(\theta - 1)}} - 1) \delta_{k-1}^{\frac{3\theta - 1}{2}} \\
	&\ge \frac{2}{1-3\theta} (2^{\frac{3\theta - 1}{3(\theta - 1)}} - 1) \delta_{k_0}^{\frac{3\theta - 1}{2}}.
	\end{align}
	Combining the above two cases and defining $C := \min \{\frac{1}{2}, \frac{2}{1-3\theta} (2^{\frac{3\theta - 1}{3(\theta - 1)}} - 1) \delta_{k_0}^{\frac{3\theta - 1}{2}} \}$, we conclude that for all $k\ge k_0$
	\begin{align}
	\phi(\delta_{k}) - \phi(\delta_{k-1}) \ge C,
	\end{align}
	which further implies that
	\begin{align}
	\phi(\delta_{k}) \ge \sum_{i=k_0 + 1}^{k} \phi(\delta_{i}) - \phi(\delta_{i-1}) \ge C(k-k_0).
	\end{align}
	Substituting the form of $\phi$ into the above inequality and simplifying the expression yields $\delta_{k} \le (\frac{2}{C(1-3\theta)(k-k_0)})^{\frac{2}{1-3\theta}}$. It follows that $r_k \le (\frac{C_3}{k-k_0})^{\frac{2}{1-3\theta}}$ for some $C_3>0$.
	
\end{proof}

\section*{Proof of \Cref{thm: finite length}}
\thmfinitelen*
\begin{proof}
	Recall the definition that $r_k:= f(\xb_k) - \bar{f}$, where $\bar{f}$ is the finite limit of $\{f(\xb_{k}) \}_k$. Also, recall that $k_0\in \mathds{N}$ is a sufficiently large integer. Then, for all $k\ge k_0$, the \KL property implies that
	\begin{align}
	\varphi' (r_k) \ge \frac{1}{\|\nabla f(\xb_k)\|} \ge \frac{2}{(L+M)\|\xb_{k} - \xb_{k-1}\|^2}, \label{eq: supp4}
	\end{align}
	where the last inequality uses the dynamics of CR in \Cref{table: 1}. Note that $\varphi(t) = \frac{c}{\theta} t^{\theta}$ is concave for $\theta \in (0,1]$. Then, by concavity we obtain that 
	\begin{align}
	\varphi(r_k) - \varphi(r_{k+1}) \ge \varphi' (r_k) (r_k - r_{k+1}) \ge \frac{M}{6(L+M)}\frac{\|\xb_{k+1} - \xb_{k}\|^3}{\|\xb_{k} - \xb_{k-1}\|^2}, \label{eq: 7}
	\end{align}
	where the last inequality uses \cref{eq: supp4} and the dynamics of CR in \Cref{table: 1}. 
	Rearranging the above inequality, taking cubic root and summing over $k= k_0,\ldots, n$ yield that (all constants are absorbed in $C$)
	\begin{align}
	\sum_{k=k_0}^{n} \|\xb_{k+1} - \xb_{k}\| &\le C \sum_{k=k_0}^{n} (\varphi(r_k) - \varphi(r_{k+1}))^{\frac{1}{3}} \|\xb_k - \xb_{k-1}\|^{\frac{2}{3}} \\
	&\overset{(i)}{\le} C \left[\sum_{k=k_0}^{n} (\varphi(r_k) - \varphi(r_{k+1}))\right]^{\frac{1}{3}} \left[\sum_{k=k_0}^{n} \|\xb_k - \xb_{k-1}\|\right]^{\frac{2}{3}}  \\
	&\overset{(ii)}{\le} C \left[\varphi(r_{k_0})\right]^{\frac{1}{3}} \left[\sum_{k=k_0}^{n} \|\xb_{k+1} - \xb_{k}\| + \|\xb_{k_0} -\xb_{k_0-1} \|\right]^{\frac{2}{3}}, \label{eq: 6}
	\end{align}
	where (i) applies the H{\"{o}}lder's inequality and (ii) uses the fact that $\varphi \ge 0$. Clearly, we must have $\lim_{n\to \infty} \sum_{k=k_0}^{n} \|\xb_{k+1} - \xb_{k}\| < +\infty$, because otherwise the above inequality cannot hold for all $n$ sufficiently large. We then conclude that $$\sum_{k=k_0}^{\infty} \|\xb_{k+1} - \xb_{k}\| < +\infty,$$
	 and the desired result follows because $k_0$ is a fixed number.
	
\end{proof}

\section*{Proof of \Cref{thm: converge_ite}}
\thmiterate*

\begin{proof}
	We prove the theorem case by case.
	
	\textbf{Case 1: $\theta = 1$.} 
	
	We have shown in case 1 of \Cref{thm: 2} that $f(\xb_k) \downarrow \bar{f}$ within finite number of iterations, i.e., $f(\xb_{k+1}) - f(\xb_k) = 0$ for all $k\ge k_0$. Based on this observation, the dynamics of CR in \Cref{table: 1} further implies that for all $k\ge k_0$
	\begin{align}
		0 = f(\xb_{k+1}) - f(\xb_{k}) \le -\frac{M}{12} \|\xb_{k+1} - \xb_{k}\|^3 \le 0.
	\end{align}
	Hence, we conclude that $\xb_{k+1} = \xb_{k}$ for all $k\ge k_0$, i.e.,	$\xb_k$ converges within finite number of iterations. Since \Cref{thm: finite length} shows that $\xb_k$ converges to some $\bar{\xb}$, the desired conclusion follows.
	
	\textbf{Case 2: $\theta \in (\frac{1}{3}, 1)$.} 
	
	Denote $\Delta_k := \sum_{i=k}^{\infty} \|\xb_{i+1} - \xb_{i} \|$. Note that \Cref{thm: finite length} shows that $\xb_{k} \to \bar{\xb}$. Thus, we have $\|\xb_k - \bar{\xb}\| \le \Delta_k$. Next, we derive the convergence rate of $\Delta_k$.
	
	By \Cref{thm: finite length}, $\lim_{n\to \infty} \sum_{i=k}^{n} \|\xb_{i+1} - \xb_{i}\|$ exists for all $k$. Then, we can let $n \to \infty$ in \cref{eq: 6} and obtain that for all $k \ge k_0$
	\begin{align}
	\Delta_k \le C[\varphi(r_k)]^{\frac{1}{3}} \Delta_{k-1}^{\frac{2}{3}} \le C r_k^{\frac{\theta}{3}} \Delta_{k-1}^{\frac{2}{3}} \overset{(i)}{\le} C(\Delta_{k-1} - \Delta_k)^{\frac{2\theta}{3(1-\theta)}} \Delta_{k-1}^{\frac{2}{3}} \le C\Delta_{k-1}^{\frac{2\theta}{3(1-\theta)} + \frac{2}{3}},
	\end{align}
	where $C$ denotes a universal constant that may vary from line to line, and (i) uses the \KL property and the dynamics of CR, i.e., $r_k \le C \|\nabla f(\xb_k)\|^{\frac{1}{1-\theta}} \le C \|\xb_k - \xb_{k-1}\|^{\frac{2}{1-\theta}}$. Note that in this case we have $\frac{2\theta}{3(1-\theta)} + \frac{2}{3} > 1$, and hence the above inequality implies that $\Delta_k$ converges to zero super-linearly as
	\begin{align}
		\Delta_k \le C^{k-k_0} \Delta_{k_0}^{(\frac{2\theta}{3(1-\theta)} + \frac{2}{3})^{k-k_0}} = \Theta \Bigg( \exp \bigg(-\bigg(\frac{2\theta}{3(1-\theta)} + \frac{2}{3}\bigg)^{k-k_0}\bigg) \Bigg).
	\end{align}
	 Since $\|\xb_k - \bar{\xb}\| \le \Delta_k$, it follows that $\|\xb_k - \bar{\xb}\|$ converges to zero super-linearly as desired.
	 %(Note that $\Delta_{k_0}$ is sufficiently small, and hence the effect of $C^{k-k_0}$ can be ignored by the super-linear diminishing term $\Delta_{k_0}^{(\frac{2\theta}{3(1-\theta)} + \frac{2}{3})^{k-k_0}}$).

	\textbf{Cases 3 \& 4.}
	
	We first derive another estimate on $\Delta_k$ that generally holds for both cases 3 and 4, and then separately consider cases 3 and 4, respectively. 
	
	Fix $\gamma \in (0,1)$ and consider $k\ge k_0$. Suppose that $\|\xb_{k+1} - \xb_{k}\| \ge \gamma \|\xb_{k} - \xb_{k-1}\|$, then \cref{eq: 7} can be rewritten as 
	\begin{align}
	\|\xb_{k+1} - \xb_{k}\| \le \frac{C}{\gamma^2} (\varphi(r_k) - \varphi(r_{k+1}))
	\end{align}
	for some constant $C>0$.
	Otherwise, we have $\|\xb_{k+1} - \xb_{k}\| \le \gamma \|\xb_{k} - \xb_{k-1}\|$. Combing these two inequalities yields that
	\begin{align}
	\|\xb_{k+1} - \xb_{k}\| \le \gamma \|\xb_{k} - \xb_{k-1}\| + \frac{C}{\gamma^2} (\varphi(r_k) - \varphi(r_{k+1})).
	\end{align}
	Summing the above inequality over $k = k_0,\ldots, n$ yields that
	\begin{align}
	\sum_{k=k_0}^n \|\xb_{k+1} - \xb_{k}\| &\le \gamma \sum_{k=k_0}^n \|\xb_{k} - \xb_{k-1}\| + \frac{C}{\gamma^2} (\varphi(r_{k_0}) - \varphi(r_{n+1})) \\
	&\le \gamma \left[\sum_{k=k_0}^n \|\xb_{k+1} - \xb_{k}\| + \|\xb_{k_0} - \xb_{k_0-1}\|\right] + \frac{C}{\gamma^2}\varphi(r_{k_0}).
	\end{align}
	Rearranging the above inequality yields that
	\begin{align}
	\sum_{k=k_0}^n \|\xb_{k+1} - \xb_{k}\| \le \frac{\gamma}{1 - \gamma} \|\xb_{k_0} - \xb_{k_0-1}\| + \frac{C}{\gamma^2(1 - \gamma)}\varphi(r_{k_0}).
	\end{align}
	Recall $\Delta_k := \sum_{i=k}^{\infty} \|\xb_{i+1} - \xb_{i} \| < + \infty$. Letting $n\to \infty$ in the above inequality yields that for all sufficiently large $k$
	\begin{align}
	\Delta_{k} &\le \frac{\gamma}{1 - \gamma} (\Delta_{k - 1} - \Delta_{k}) + \frac{C}{\gamma^2(1 - \gamma)\theta}r_{k}^\theta \\
	&\overset{(i)}{\le} \frac{\gamma}{1 - \gamma} (\Delta_{k - 1} - \Delta_{k}) + \frac{C}{\gamma^2(1 - \gamma)\theta}\|\xb_{k} - \xb_{k-1}\|^{\frac{2\theta}{1 - \theta}} \\
	&\le \frac{\gamma}{1 - \gamma} (\Delta_{k - 1} - \Delta_{k}) + \frac{C}{\gamma^2(1 - \gamma)\theta}(\Delta_{k - 1} - \Delta_{k})^{\frac{2\theta}{1 - \theta}}, \label{eq: 8}
	\end{align}
	where (i) uses the \KL property and the dynamics of CR, i.e., $r_k \le C \|\nabla f(\xb_k)\|^{\frac{1}{1-\theta}} \le C \|\xb_k - \xb_{k-1}\|^{\frac{2}{1-\theta}}$. 
	
	\textbf{Case 3:} $\theta = \frac{1}{3}$. In this case, $\frac{2\theta}{1 - \theta} = 1$ and \cref{eq: 8} implies that $\Delta_{k} \le C (\Delta_{k - 1} - \Delta_{k})$ for all sufficiently large $k$, i.e., $\Delta_k$ converges to zero linearly as $\Delta_{k} \le (\frac{C}{1+C})^{k-k_0} \Delta_{k_0}$. The desired result follows since $\|\xb_{k} - \bar{\xb}\| \le \Delta_k$.
	
	\textbf{Case 4:} $\theta \in (0, \frac{1}{3})$. In this case, $0<\frac{2\theta}{1 - \theta} < 1$ and \cref{eq: 8} can be asymptotically rewritten as $\Delta_k \le \frac{C}{\gamma^2(1 - \gamma)\theta}(\Delta_{k - 1} - \Delta_{k})^{\frac{2\theta}{1 - \theta}}$. This further implies that
	\begin{align}
	\Delta_k^{\frac{1-\theta}{2\theta}} \le C (\Delta_{k-1} - \Delta_k)
	\end{align}
	for some constant $C>0$.
	Define $h(t) = t^{-\frac{1-\theta}{2\theta}}$ and fix $\beta > 1$. Suppose first that $h(\Delta_k) \le \beta h(\Delta_{k-1})$. Then the above inequality implies that
	\begin{align}
	1 &\le C \frac{\Delta_{k-1} - \Delta_k}{\Delta_k^{\frac{1-\theta}{2\theta}}} = C(\Delta_{k-1} - \Delta_k) h(\Delta_k) \le C\beta(\Delta_{k-1} - \Delta_k) h(\Delta_{k-1}) \\
	&\le C\beta \int_{\Delta_k}^{\Delta_{k-1}} h(t) dt = C\beta \frac{2\theta}{3\theta -1} (\Delta_{k-1}^{\frac{3\theta-1}{2\theta}} - \Delta_{k}^{\frac{3\theta-1}{2\theta}}).
	\end{align}
	Set $\mu:= \frac{1-3\theta}{2C\beta \theta} > 0, \nu := \frac{3\theta - 1}{2\theta} < 0$. Then the above inequality can be rewritten as
	\begin{align}
	\Delta_k^\nu - \Delta_{k-1}^\nu \ge \mu.
	\end{align}
	Now suppose $h(\Delta_k) > \beta h(\Delta_{k-1})$, which implies that $\Delta_{k} < q\Delta_{k-1}$ with $q = \beta^{-\frac{2\theta}{1-\theta}} \in (0,1)$. Then, we conclude that $\Delta_k^\nu \ge q^\nu \Delta_{k-1}^\nu$ and hence $\Delta_k^\nu - \Delta_{k-1}^\nu \ge (q^\nu - 1) \Delta_{k-1}^\nu$. Since $q^\nu - 1 > 0$ and $\Delta_{k-1}^\nu \to +\infty$, there must exist $\bar{\mu}>0$ such that $(q^\nu - 1) \Delta_{k-1}^\nu \ge \bar{\mu}$ for all sufficiently large $k$. Thus, we conclude that $\Delta_k^\nu - \Delta_{k-1}^\nu \ge \bar{\mu}$. Combining two cases, we obtain that for all sufficiently large $k$,
	\begin{align}
	\Delta_k^\nu - \Delta_{k-1}^\nu \ge \min \{\mu, \bar{\mu}\}.
	\end{align}
	Telescoping the above inequality over $k = k_0, \ldots, k$ yields that 
	\begin{align}
	\Delta_k \le [\Delta_{k_0}^\nu + \min \{\mu, \bar{\mu}\} (k-k_0)]^{\frac{1}{\nu}} \le \bigg(\frac{C}{k-k_0}\bigg)^{\frac{2\theta}{1-3\theta}},
	\end{align}
	where $C$ is a certain positive constant.
	The desired result then follows from the fact that $ \|\xb_{k} - \bar{\xb}\| \le \Delta_k$.
\end{proof}

\section*{Proof of \Cref{prop: 1}}
\propeb*

\begin{proof}
	The proof idea follows from that in \cite{Yue2018}.	
	Consider any $\xb\in \Omega^c \cap \{\xb\in \RR^d : \dist_\Omega(\xb)<\varepsilon, f_\Omega < f(\xb) <f_\Omega + \lambda\}$, and consider the following differential equation
	\begin{align}\label{eq: 1}
	\ub(0) = \xb, \quad \overset{\bigcdot}{\ub}(t) = -\nabla f(\ub(t)), \quad \forall t>0.
	\end{align}
	As $\nabla f$ is continuously differentiable, it is Lipschitz on every compact set. Thus, by the Picard-Lindel\"{o}f theorem \cite[Theorem II.1.1]{Hartman}, there exists $\nu > 0$ such that \cref{eq: 1} has a unique solution $\ub_\xb (t)$ over the interval $[0, \nu]$. Define $\Delta(t):= f(\ub_\xb (t)) - f_\Omega$. 
	Note that $\Delta(t)>0$ for $t\in [0,\nu]$, as otherwise there exists $\hat{t} \in [0, \nu]$ such that $\ub_\xb (\hat{t}) \in \Omega$ and hence $\ub_\xb \equiv \ub_\xb (\hat{t}) \in \Omega$ is the unique solution to \cref{eq: 1}. This contradicts the fact that $\ub(0) \in \Omega^c$. 
	
	Using \cref{eq: 1} and the chain rule, we obtain that for all $t\in [0, \nu]$
	\begin{align}
	\overset{\bigcdot}{\Delta}(t) = \inner{\nabla f(\ub_\xb (t))}{\overset{\bigcdot}{\ub}_\xb(t)} = - \|\nabla f(\ub_\xb (t))\| \|\overset{\bigcdot}{\ub}_\xb(t)\|.
	\end{align} 
	Applying the \KL property in \cref{eq: KLsimple} to the above equation yields that
	\begin{align}\label{eq: 2}
	\overset{\bigcdot}{\Delta}(t) \le - \left( \frac{\Delta(t)}{C}\right)^{1-\theta} \|\overset{\bigcdot}{\ub}_\xb(t)\|,
	\end{align}
	where $C>0$ is a certain universal constant.
	Since $\Delta(t) > 0$, \cref{eq: 2} can be rewritten as
	\begin{align}
	\|\overset{\bigcdot}{\ub}_\xb(t)\| \le -\frac{C^{1-\theta}}{\theta} (\Delta(t)^\theta)'.
	\end{align}
	Based on the above inequality, for any $0\le a<b<\nu$ we obtain that
	\begin{align}
	\|\ub_\xb(b) - \ub_\xb(a)\| &= \|\int_{a}^{b} \overset{\bigcdot}{\ub}_\xb(t) dt\| \le \int_{a}^{b} \|\overset{\bigcdot}{\ub}_\xb(t) \|dt \nonumber\\
	&\le - \int_{a}^{b} \frac{C^{1-\theta}}{\theta} [\Delta(t)^\theta]' dt = \frac{C^{1-\theta}}{\theta} [\Delta(a)^\theta - \Delta(b)^\theta]. \label{eq: 3}
	\end{align}
	In particular, setting $a=0$ in \cref{eq: 3} and noting that $\ub_\xb(0) = \xb$, we further obtain that
	\begin{align}
	\|\ub_\xb(b) - \xb\| \le \frac{C^{1-\theta}}{\theta} (f(\xb) - f_\Omega)^\theta. \label{eq: supp1}
	\end{align}
	
	Next, we show that $\nu = +\infty$. Suppose $\nu < +\infty$, then \cite[Corollary II.3.2]{Hartman} shows that $\|\ub_\xb(t)\| \to +\infty$ as $t \to \nu$. However, 
	\cref{eq: supp1} implies that
	\begin{align*}
	\|\ub_\xb(t)\| \le \|\xb\| + \|\ub_\xb(t) - \xb\| \le \|\xb\| + \frac{C^{1-\theta}}{\theta}(f(\xb) - f_\Omega)^\theta < +\infty,
	\end{align*}
	which leads to a contradiction. Thus, $\nu = +\infty$. 
	
	Since $\overset{\bigcdot}{\Delta}(t) \le 0$, $\Delta(t)$ is non-increasing. Hence, the nonnegative sequence $\{\Delta(t)\}$ has a limit. Then, \cref{eq: 3} further implies that $\{\ub_\xb(t)\}$ is a Cauchy sequence and hence has a limit $\ub_\xb(\infty)$. Suppose $\nabla f(\ub_\xb(\infty)) \ne \zero$. Then we obtain that $\lim_{t \to \infty}\overset{\bigcdot}{\Delta}(t) = - \|\nabla f(\ub_\xb(\infty))\|^2 < 0$, which contradicts the fact that $\lim_{t\to \infty} \Delta(t)$ exists. Thus, $\nabla f(\ub_\xb(\infty)) = \zero$, and this further implies that $\ub_\xb(\infty)\in \Omega, f(\ub_\xb(\infty)) = f_\Omega$ by the \KL property in \cref{eq: KLsimple}. We then conclude that 
	\begin{align}
	\dist_\Omega(\xb) \le \|\xb - \ub_\xb(\infty)\| = \lim_{t \to \infty} \|\xb - \ub_\xb(t)\| \le \frac{C^{1-\theta}}{\theta}(f(\xb) - f_\Omega)^\theta.
	\end{align}
Combining the above inequality with the \KL property in \cref{eq: KLsimple}, we obtain the desired \KL error bound.
\end{proof}

\section*{Proof of \Cref{thm: dist}}
\thmdist*

\begin{proof}
   We prove the theorem case by case.
   
	\textbf{Case 1: $\theta = 1$.} 
	
	We have proved in \Cref{thm: converge_ite} that $\xb_k \to \bar{\xb}\in \Omega$ within finite number of iterations. 
	Since $\dist_\Omega(\xb_k) \le \|\xb_k - \bar{\xb}\|$, we conclude that $\dist_\Omega(\xb_k)$ converges to zero within finite number of iterations. 
	
	\textbf{Case 2: $\theta \in (\frac{1}{3}, 1)$.} 
	
	By the \KL error bound in \Cref{prop: 1}, we obtain that
	\begin{align}
	\dist_\Omega(\xb_{k+1}) \le C \|\nabla f(\xb_{k+1})\|^{\frac{\theta}{1-\theta}} \le C\|\xb_{k+1} - \xb_k\|^{\frac{2\theta}{1-\theta}}, \label{eq: supp5}
	\end{align}
	where the last inequality uses the dynamics of CR in \Cref{table: 1}. On the other hand, \cite[Lemma 1]{Yue2018} shows that 
	\begin{align}
		\|\xb_{k+1} - \xb_k\| \le C \dist_\Omega(\xb_k). \label{eq: supp6}
	\end{align}
  Combining \cref{eq: supp5} and \cref{eq: supp6} yields that
	\begin{align}
	\dist_\Omega(\xb_{k+1}) \le C \dist_\Omega(\xb_{k})^{\frac{2\theta}{1-\theta}}.
	\end{align}
	Note that in this case we have $\frac{2\theta}{1-\theta} > 1$. Thus, $\dist_\Omega(\xb_{k})$ converges to zero super-linearly as desired.
	
	\textbf{Cases 3 \& 4: $\theta \in (0, \frac{1}{3}]$.} 
	
	Note that $\dist_\Omega(\xb_k) \le \|\xb_k - \bar{\xb}\|$. The desired results follow from Cases 3 \& 4 in
	\Cref{thm: converge_ite}.
	
\end{proof}

\section*{Proof of \Cref{coro: mu}}
\thmmu*

\begin{proof}

By the dynamics of CR in \Cref{table: 1}, we obtain that
\begin{align}
\|\nabla f(\xb_{k+1})\| &\le \frac{L+M}{2} \|\xb_{k+1} - \xb_{k}\|^2, \\
-\lambda_{\min} (\nabla^2 f(\xb_{k+1})) &\le \frac{2L + M}{2} \|\xb_{k+1} - \xb_{k}\|.
\end{align}
The above two inequalities imply that $\mu(\xb_{k}) \le \|\xb_{k+1} - \xb_{k}\|$. Also, \cite[Lemma 1]{Yue2018} shows that $\|\xb_{k+1} - \xb_{k}\|\le C \dist_\Omega(\xb_{k})$. Then, the desired convergence result for $\mu(\xb_{k})$ follows from \Cref{thm: dist}.

\end{proof}

\end{document}